\documentclass[11pt,reqno]{amsart}

\usepackage{amssymb,amsmath,amsthm,amscd,latexsym,amsfonts}
\usepackage{mathtools}
\usepackage[T1]{fontenc}

\usepackage{graphicx}
\usepackage{color}
\usepackage[a4paper,top=3cm, bottom=3cm, left=3cm, right=2cm]{geometry}

\newtheorem{thm}{Theorem}
\newtheorem{defn}{Definition}

\newtheorem{pro}{Proposition}
\newtheorem{rk}{Remark}

\numberwithin{equation}{section} 

\newcommand{\M}{{\mathcal M}}

\newcommand{\bea}{\begin{eqnarray}}
\newcommand{\eea}{\end{eqnarray}}

\newcommand{\R}{\mathbb{R}}

\def\M{\mathcal M}

\def\R{\mathbb{R}}

\begin{document}
\title [Construction of flows of finite-dimensional algebras]
{Construction of flows of finite-dimensional algebras}

\author{M. Ladra,  U.A. Rozikov}

\address{M.\ Ladra\\ Department of Mathematics, University of Santiago de Compostela, 15782, Spain.}
 \email {manuel.ladra@usc.es}

 \address{U.\ A.\ Rozikov \\ Institute of mathematics, 29, Do'rmon Yo'li str., 100125,
Tashkent, Uzbekistan.} \email {rozikovu@yandex.ru}

\begin{abstract} Recently, we introduced the
notion of flow (depending on time) of finite-dimensional
algebras. A flow of algebras (FA) is a particular case of a continuous-time dynamical system
whose states are finite-dimensional algebras with (cubic) matrices of
structural constants satisfying an analogue
of  the Kolmogorov-Chapman equation (KCE).
Since there are several kinds of multiplications between cubic matrices one has fix a multiplication
first and then consider the KCE with respect to the fixed multiplication.
The existence of a solution for the KCE provides the existence of an FA.
In this paper our aim is to find sufficient conditions on
the multiplications under which the corresponding KCE has a solution.
Mainly our conditions  are given on the algebra of cubic matrices (ACM) considered
with respect to a fixed multiplication of cubic matrices.
Under some assumptions on the  ACM (e.g. power associative, unital, associative, commutative) we describe
a wide class of FAs, which contain algebras of arbitrary finite dimension.
In particular,  adapting the theory of continuous-time Markov processes,
 we construct a class of FAs given by  the matrix exponent of cubic matrices.
 Moreover, we remarkably extend the set of FAs given with respect to the
 Maksimov's multiplications of our paper \cite{LR}.
For several FAs we study the time-dependent behavior (dynamics) of the algebras.
We derive a system of differential equations for FAs.
\end{abstract}



\maketitle

\section{Introduction}

 It is known that (see e.g. \cite{SK}) if each element of a
 family (depending on time) of matrices satisfying
the Kolmogorov-Chapman  equation (KCE) is stochastic,
then it generates a Markov process.
But what kind of process or dynamical systems can be generated by a family of
non-stochastic matrices satisfying KCE? Depending on the matrices, it can be a
non-Markov process \cite{ht}, a deformation \cite{O}, etc.
Other motivations of consideration of non-stochastic solutions of KCE are given
in recent papers \cite{CLR,Mu,ORT,RM,RM1}. These papers devoted to study some chains of evolution algebras.
In each of these papers the matrices of structural constants (time-dependent on  the pair $(s, t)$) are square or rectangular and satisfy the
KCE. In other words, a chain of evolution algebras is a continuous-time dynamical
system which in any fixed time is an evolution algebra.

In \cite{DR}, \cite{LR13}
some cubic stochastic matrices
are used (as matrices of structural constants)
to investigate algebras and dynamical systems of bisexual populations.

In \cite{LR} we generalized the notion of chain of evolution algebras (given for algebras with \emph{rectangular} matrices)
to a notion of flow of arbitrary finite-dimensional algebras (i.e.
their matrices of structural constants are \emph{cubic} matrices).
In this paper we continue our investigations of flows of algebras (FAs).

The paper is organized as follows. In Section~\ref{S:definitions} we give the main definitions related to
algebras of cubic matrices, several kinds of multiplications of cubic matrices and FAs.
Note that an FA is defined by a family (depending on time) of cubic matrices of structural
 constants, which satisfy an analogue of KCE. Since there are several types of multiplication of cubic matrices,
one has to fix a multiplication, say $\mu$, first and then consider the KCE
with respect to this multiplication. In Section~\ref{Se2} we find some conditions on $\mu$ under which the
KCE has at least one solution. For several multiplications we give a wide class of solution of KCE, i.e. a wide class
of FAs. For the multiplications of Maksimov \cite{Mak} we extend the class of FAs given in \cite{LR}. Section~\ref{Sd}
contains some differential equations for FAs.
For several FAs we study the time-dependent behavior of the flows.

\section{Definitions}\label{S:definitions}

\subsection{Algebras of cubic matrices} Given a field $F$, any finite-dimensional algebra $\mathcal A$ can be specified up
to isomorphism by giving its dimension (say $m$), and specifying $m^3$ structural
constants $c_{ijk}$, which are scalars in $F$.
These structure constants determine the multiplication in $\mathcal A$ via the following rule:
\[{e}_{i} {e}_{j} = \sum_{k=1}^m c_{ijk} {e}_{k},\]
where $e_1,\dots,e_m$ form a basis of $\mathcal A$.
Thus the multiplication of a finite-dimensional algebra is given by a cubic matrix $(c_{ijk})$.

A cubic matrix $Q=(q_{ijk})_{i,j,k=1}^m$ is a $m^3$-dimensional vector which can be uniquely written as
\[Q=\sum_{i,j,k=1}^m   q_{ijk}E_{ijk},\]
where $E_{ijk}$ denotes the cubic unit (basis) matrix, i.e. $E_{ijk}$ is a $m^3$- cubic matrix whose
$(i,j,k)$th entry is equal to 1 and all other entries are equal to 0.

Denote by $\mathfrak C$ the set of all cubic matrices over a field $F$. Then $\mathfrak C$ is an
$m^3$-dimensional vector space over $F$, i.e. for any matrices $A=(a_{ijk}), B=(b_{ijk})\in \mathfrak C$, $\lambda\in F$, we have
\[ A+B \coloneqq (a_{ijk}+b_{ijk})\in \mathfrak C, \qquad  \lambda A \coloneqq (\lambda a_{ijk})\in \mathfrak C. \]
In general, one can fix an $m^3\times m^3\times m^3$- cubic matrix $\mu=\left(C_{ijk,lnr}^{uvw}\right)$ as a matrix of structural constants and   give a \emph{multiplication} of basis cubic matrices as
\begin{equation}\label{uk}
E_{ijk}*_\mu E_{lnr}=\sum_{uvw}C_{ijk,lnr}^{uvw}E_{uvw}.
\end{equation}
Then the extension of this multiplication by bilinearity to arbitrary cubic matrices gives a general
multiplication on the set $\mathfrak C$ and it becomes an algebra of cubic matrices (ACM), denoted by $\mathfrak C_\mu$.
Under known conditions (see \cite{Ja}) on structural constants one can make this general ACM as a commutative or/and associative algebra, etc.

\subsection{Maksimov's multiplications} Here we introduce some simple versions of multiplications \eqref{uk}.
Denote $I=\{1,2,\dots,m\}$.
Following  \cite{Mak} define the following multiplications for basis matrices $E_{ijk}$:
\begin{equation}\label{ma}
 E_{ijk}*_a E_{lnr}=\delta_{kl}E_{ia(j,n)r},
 \end{equation}
  where $a \colon   I\times I\to I$, $ (j,n) \mapsto a(j,n) \in I$,  is an arbitrary associative binary operation and
  $\delta_{kl}$ is the Kronecker symbol.

  Denote by $\mathcal O_m$ the set of all associative binary operations on $I$.

 The general formula for the multiplication is the extension of \eqref{ma} by bilinearity, i.e.
 for any two cubic matrices $A=(a_{ijk}), B=(b_{ijk})\in \mathfrak C$
 the matrix $A*_a B=(c_{ijk})$ is defined by
 \begin{equation}\label{AB}
 c_{ijr}=\sum_{l,n: \, a(l,n)=j}\sum_k a_{ilk}b_{knr}.
 \end{equation}

 Denote by $\mathfrak C_a\equiv \mathfrak C_a^m=(\mathfrak C, *_a) $, $a\in\mathcal O_m$,
 the ACM given by the multiplication $*_a$.

\subsection{Flow of algebras}
Following \cite{LR} we define a notion of flow of algebras (FA).
Consider a family $\left\{\mathcal A^{[s,t]}:\ s,t \in \R,\ 0\leq s\leq t
\right\}$ of arbitrary $m$-dimensional algebras over the field $F$,
with basis $e_1,\dots,e_m$ and multiplication table
\[  e_ie_j =\sum_{k=1}^mc_{ijk}^{[s,t]}e_k, \ \ i,j=1,\dots,m. \]
 Here parameters $s,t$ are
considered as time.

Denote by
$\M^{[s,t]}=\left(c_{ijk}^{[s,t]}\right)_{i,j,k=1,\dots,m}$ the matrix
of structural constants of $\mathcal A^{[s,t]}$.

\begin{defn}
 Fix an arbitrary multiplication (not necessarily the Maksimov's multiplication) of cubic matrices, say $*_\mu$.

 A family $\left\{\mathcal A^{[s,t]}:\ s,t \in \R,\ 0\leq s\leq t
\right\}$ of $m$-dimensional algebras over the field $F$
is called an FA of type $\mu$ if the matrices
$\M^{[s,t]}$ of structural constants satisfy the
Kolmogorov-Chapman equation (for cubic matrices):
\begin{equation}\label{KC}
\M^{[s,t]}=\M^{[s,\tau]}*_\mu\M^{[\tau,t]}, \qquad \text{for all} \ \ 0\leq s<\tau<t.
\end{equation}
\end{defn}

See \cite{LR} for useful remarks and several examples of FAs.

\begin{defn}
 An FA is called a (time) homogenous FA if
the matrix $\M^{[s,t]}$ depends only on $t-s$. In this case we write
$\M^{[t-s]}$.
\end{defn}
\begin{defn}
 An FA is called a discrete-time FA if
the matrix $\M^{[s,t]}$ depends only on $t,s\in \mathbb N$. In this case we write
$\M^{[n,m]}$, $n,m\in \mathbb N$.
\end{defn}
To construct an FA of type $\mu$ one has to solve \eqref{KC}. In this paper our aim is to construct FAs,
i.e. find solutions to the equation \eqref{KC}.

\section{Constructions and time dynamics of FAs}\label{Se2}

\subsection{Power-associative multiplications} Recall that an algebra is called \emph{power-associative} if for each element $a$
the subalgebra generated by $a$ is associative, that is $a^na^m=a^{m+n}$,
for all nonnegative integers $n, m$.

Fix an arbitrary multiplication rule between cubic matrices, say $*_\mu$.

For a cubic matrix $Q$ denote
\[Q^{*_\mu n}=\underbrace{Q*_\mu Q*_\mu\dots *_\mu Q}_{n \ \ \text{times}}, \qquad  n=1,2,\dots.\]

\textsl{Condition 1}. Assume that the algebra $\mathfrak C_\mu$ of cubic matrices is power-associative.

\begin{pro}\hfill
 \begin{itemize}
\item[1.] If Condition 1 is satisfied
for  the multiplication $*_\mu$ then  $\M^{[n,m]}=Q^{*_\mu (m-n)}$ generates a discrete time FA of type $\mu$.
\item[2.] If  ACM $\mathfrak C_\mu$ has an idempotent, i.e. there exists $\mathcal I\in \mathfrak C_\mu$
such that $\mathcal I*_\mu \mathcal I=\mathcal I$ then  $\M^{[s,t]}=\mathcal I$, for all $0\leq s<t$, generates an FA of type $\mu$.
\end{itemize}
\end{pro}
\begin{proof}
1. Using power-associativity of the ACM
for any natural numbers $n<k<m$ we get
\[Q^{*_\mu (m-n)}=Q^{*_\mu (k-n)}*_\mu Q^{*_\mu (m-k)}.\]
Thus $\M^{[n,m]}=Q^{*_\mu (m-n)}$ is a solution to \eqref{KC} and therefore it
generates a discrete-time FA of type $\mu$, denote it by $\mathcal A^{[n,m]}_1$.

2.  Straightforward. Denote the corresponding FA by $\mathcal A^{[s,t]}_2$.
\end{proof}

\textsl{Remarks on time dynamics}.  The time dynamics of $\mathcal A_1^{[n,m]}$ depends on the fixed matrix $Q$.
This FA is a time homogenous; it can be periodic iff the powers (with respect to multiplication $\mu$) of
the cubic matrix $Q$ have periods, i.e. if there is $p\in \mathbb{N}$ such that $ Q^{*_\mu p}=Q$; if $Q$ is such that
$\lim_{n\to \infty} Q^{*_\mu n}=\tilde Q$ exists then the FA has a limit algebra:
\[\tilde{\mathcal A}_1=\lim_{m-n\to\infty} \mathcal A_1^{[n,m]}.\]

The time dynamics of $\mathcal A_2^{[s,t]}$ is trivial: it does not depend on time. \\

\subsection{Exponential solutions}
In the theory of continuous-time Markov chains, the matrix exponent of square matrices plays a crucial role (see e.g. \cite[Chapter 3]{BB}, \cite[Chapter 2]{SK}).
Here we adapt this notion of matrix exponent for cubic matrices.

 Recall that an algebra is \emph{unital} if it has an
element $u$ with $ux = x = xu$ for all $x$ in the algebra. The element $u$ is called \emph{unit} element.

\textsl{Condition 2}. Assume $*_\mu$ is such that the  corresponding ACM,
$\mathfrak C_\mu=(\mathfrak C, *_\mu)$ is unital, i.e. it has a unit matrix denoted by ${\mathbb I}$.\\

In the unital ACM $\mathfrak C_\mu$ for each cubic matrix $Q$ we define $Q^{*_\mu 0}={\mathbb I}$.

A \emph{matrix exponent} $\exp_\mu(tQ)$ is defined by
\begin{equation}\label{exp}
\exp_\mu(tQ)={\mathbb I}+\sum_{n\geq 1}\frac{(tQ)^{*_\mu n}}{n!}=\sum_{n\geq 0}\frac{(tQ)^{*_\mu n}}{n!}, \ \ \text{i.e.} \ \
 \big(\exp_\mu(tQ)\big)_{ijk}=\sum_{n\geq 0}\frac{t^n(Q^{*_\mu n})_{ijk}}{n!}.
\end{equation}
For a  cubic matrix $Q$, the parameter $t$ in \eqref{exp} can be any real number. In our paper we consider the case $t\geq 0$.

\textsl{Condition 3}. Assume that $\mathfrak C_\mu$ is a normed space with norm $\lVert \cdot\rVert$ such that for any
two cubic matrices $A$, $B$, we have
 \begin{equation}\label{ne}
 \lVert A*_\mu B\rVert \leq\lVert A\rVert \lVert B\rVert.
 \end{equation}
 
\textsl{Condition 4}. Assume $\mathfrak C_\mu$ is an associative algebra.

\begin{pro} Assume Conditions 1--4 are satisfied. Let $Q$ be a cubic matrix. Then
\begin{itemize}
\item[(a)] The series in \eqref{exp} converges.

\item[(b)]  The function $\exp_\mu(tQ)$ is differentiable with respect to $t$ with
\[\frac{d}{dt}\exp_\mu(tQ)=Q*_\mu \exp_\mu(tQ)=\exp_\mu(tQ)*_\mu Q, \ \ t\in \mathbb R.\]

\item[(c)] The semigroup property:
\begin{equation}\label{ets}
\exp_\mu((s+t)Q)=\exp_\mu(sQ)*_\mu\exp_\mu(tQ), \ \ \text{for all} \ s,t\in \mathbb R.
\end{equation}

\item[(d)] The function $\exp_\mu(tQ)$ is the only solution to
\begin{equation}\label{Y}
\frac{d}{dt}Y(t)=Y(t)*_\mu Q=Q*_\mu Y(t), \ \ \text{with} \ \ Y(0)={\mathbb I}.
\end{equation}
These are called the Kolmogorov's forward and backward equations.
\end{itemize}
\end{pro}
\begin{proof}
(a)
 Using properties of the norm and \eqref{ne} we get
\[ \lVert\exp_\mu(tQ)\rVert=\big\lVert\sum_{n\geq 0}\frac{(tQ)^{*_\mu n}}{n!} \big\rVert \leq \sum_{n\geq 0}\frac{\lvert t \rvert ^k\lVert Q\rVert ^n}{n!}=\exp(\lvert t\rvert \, \lVert Q\rVert)<+\infty. \]

(b)  Write
\[\frac{d}{dt}\exp_\mu(tQ)=\frac{d}{dt}\sum_{n\geq 0}\frac{(tQ)^{*_\mu n}}{n!}=\sum_{n\geq 1}\frac{t^{n-1}Q^{*_\mu n}}{(n-1)!}
=Q*_\mu\sum_{n\geq 1}\frac{(tQ)^{*_\mu {n-1}}}{(n-1)!}   =Q*_\mu \exp_\mu(tQ).\]
It is obvious that
\[Q*_\mu \exp_\mu(tQ)= Q*_\mu\sum_{n\geq 1}\frac{(tQ)^{*_\mu {n-1}}}{(n-1)!}=\sum_{n\geq 1}\frac{(tQ)^{*_\mu {n-1}}}{(n-1)!}*_\mu Q=  \exp_\mu(tQ)*_\mu Q.\]

(c) From \eqref{exp} we have
\begin{align*}
\exp_\mu(sQ)*_\mu \exp_\mu(tQ) &=\left(\sum_{n\geq 0}\frac{(sQ)^{*_\mu n}}{n!}\right)*_\mu \left(\sum_{k\geq 0}\frac{(tQ)^{*_\mu k}}{k!}\right)\\
{}&=\sum_{n\geq 0}\sum_{k\geq 0}\frac{Q^{*_\mu (n+k)}s^nt^k}{n! \, k!}.
\end{align*}
Denoting $j=n+k$ from the last equality, we get
  \begin{align*}
  \exp_\mu(sQ)*_\mu \exp_\mu(tQ) &=\sum_{j\geq 0}\sum_{k\geq 0}\frac{Q^{*_\mu j}s^{j-k}t^k}{(j-k)! \, k!}=
  \sum_{j\geq 0}\frac{Q^{*_\mu j}}{j!}\sum_{k\geq 0}\frac{j!}{(j-k)! \, k!}s^{j-k}t^k \\
 {}& =\sum_{j\geq 0}\frac{Q^{*_\mu j}(s+t)^j}{j!}=\exp_\mu((s+t)Q).
   \end{align*}

(d) From part (b) it follows that $Y(t)=\exp_\mu(tQ)$ is a solution to \eqref{Y}.
We show its uniqueness. Assume $Z(t)$ is another solution.
Then take
\begin{align*}
\frac{d}{dt}\left(Z(t)*_\mu \exp_\mu(-tQ)\right)&=\frac{d}{dt}Z(t)*_\mu \exp_\mu(-tQ)+Z(t)*_\mu \frac{d}{dt}\exp_\mu(-tQ) \\
{}&=Q*_\mu Z(t)*_\mu \exp_\mu(-tQ)-Z(t)*_\mu Q*_\mu\exp_\mu(-tQ)=0.
\end{align*}
Thus $Z(t)*_\mu \exp_\mu(-tQ)$ is independent on $t$. Consequently, since at $t=0$ we have this function equal to $\mathbb I$
it follows  that $Z(t)*_\mu \exp_\mu(-tQ)=\mathbb I$, consequently by part (c) we get $ Z(t)=\exp_\mu(tQ)$.
\end{proof}

Using \eqref{ets} one easily checks that  $\M^{[s,t]}=\exp_\mu((t-s)Q)$
satisfies \eqref{KC}.

Summarizing, we get the following.
\begin{thm} Let $\mu$ be a multiplication such that Conditions 1--4 are satisfied.
Let $Q$ be a  cubic matrix.
Then the family of matrices $\left\{\M^{[s,t]}=\exp_\mu((t-s)Q)
\right\}$, $0\leq s<t$, generates a time-homogeneous FA of type $\mu$, denoted by $\mathcal A_3^{[s,t]}$.
\end{thm}

\textsl{Remarks on time dynamics}. To study time dynamics of FA $\mathcal A_3^{[s,t]}$ one need to compute
the matrix exponent $\exp_\mu(tQ)$. We note that even in  case of square matrices finding methods to
compute the matrix exponent is difficult, and this is still a topic of considerable current research.
But for some simple $Q$ one can compute the matrix exponent. For example, if $Q$ is nilpotent,
i.e. there exists $q$ such that $Q^{*_\mu q}=0$. Then we have
\[\exp_\mu(tQ)=\mathbb I+tQ+\frac{t^2}{2}Q^{*_\mu 2}+\dots+\frac{t^{q-1}}{(q-1)!}Q^{*_\mu (q-1)}.\]
For the corresponding $\mathcal A_3^{[s,t]}$ we see that if $Q$ is nilpotent then does not exist a
limit algebra at $t\to\infty$.\\

\subsection{Time non-homogenous FAs}
The following theorem gives an example of $m$-dimensional time non-homogenous FA.

Let $\mathbb I\in \mathfrak C_\mu$ be a unit matrix (under Condition 2). Matrix $A^{-1}\in \mathfrak C_\mu$
is called inverse of an $A\in \mathfrak C_\mu$ if
\[A*_\mu A^{-1}= A^{-1}*_\mu A=\mathbb I.\]

\begin{thm} Assume $*_\mu$ such that Condition 2 and 4 are satisfied.
Let $\{A^{[t]}, \,t\geq 0\}\subset \mathfrak C_\mu$ be a family of
invertible (for all $t$) $m^3$-matrices. Then the matrix
\[\M^{[s,t]}=A^{[s]}*_\mu(A^{[t]})^{-1},\]
generates an FA of type $\mu$, where $(A^{[t]})^{-1}$ is the inverse of $A^{[t]}$ .
\end{thm}
\begin{proof}
By associativity of the multiplication $*_\mu$ of matrices we get
\[\M^{[s,\tau]}*_\mu\M^{[\tau,t]}=A^{[s]}*_\mu
\left((A^{[\tau]})^{-1}*_\mu A^{[\tau]}\right)*_\mu(A^{[t]})^{-1}=A^{[s]}*_\mu(A^{[t]})^{-1}=\M^{[s,t]}.\]
Thus $\M^{[s,t]}$ satisfies \eqref{KC}, consequently  each family (with one parameter) of invertible
cubic matrices defines an FA, denoted by $\mathcal A_4^{[s,t]}$, which is time non-homogenous, in
general. But will be a time homogenous FA, for example, if
$A^{[t]}$ is equal to $t$th power of an invertible matrix $A$.
\end{proof}

\textsl{Remarks on time dynamics}. Depending on the given family $A^{[t]}$ one can study time dynamics of
the FA $\mathcal A_4^{[s,t]}$ (see  \cite[Example 4]{CLR} for a simple case).\\

\subsection{Constructions for Maksimov's multiplications}
Denote by $\mathcal O_m$ the set of all associative binary operations on $I=\{1,2,\dots,m\}$.

\begin{defn} We say that an operation $a\in \mathcal O_m$ is uniformly distributed if
\[\sum_{l,n: a(l,n)=j} 1=m,\]
independently on $j\in I$.
\end{defn}

For example, an operation $a$ with $a(i,j)=j$,  for all $i,j\in I$, is  uniformly distributed. But $a(i,j)=1$, for all $i,j\in I$, is not uniformly distributed.

\begin{thm} Consider the Maksimov's multiplication \eqref{AB} with respect to a uniformly distributed $a$.
 Take arbitrary functions $f_i(t)$ and $g_i(t)$, $i=1,\dots m$, of $t\in \mathbb R$ such that
\begin{equation}\label{abm}
\sum_{k=1}^mf_k(t)g_k(t)=\frac{1}{m}, \ \ \text{for any} \ \ t\in \mathbb R.
\end{equation}
Then the family of cubic matrices $\mathcal M^{[s,t]}=\left(f_i(s)g_k(t)\right)_{i,j,k=1}^m$ generates an FA of type $a$, denoted by $\mathcal A^{[s,t]}_5$.

\emph{Notice that the entries of these cubic matrices do not depend on the middle index $j$.}
\end{thm}
\begin{proof} Let $\mathcal M^{[s,t]}=\left(f_i(s)g_k(t)\right)_{i,j,k=1}^m$ then the equation \eqref{KC} has the form
\begin{equation}\label{aMM}
 f_i(s)g_r(t)=\sum_{l,n: \, a(l,n)=j}\sum_k  f_i(s)g_k(\tau)f_k(\tau)g_r(t).
\end{equation}
Simplify the system \eqref{aMM} to get
\[
 f_i(s)g_r(t)\left[\sum_{l,n: \, a(l,n)=j}\sum_k  f_k(\tau)g_k(\tau)-1\right]=0.
\]
By conditions of the theorem we have
\[\sum_{l,n: \, a(l,n)=j}\sum_k  f_k(\tau)g_k(\tau)-1=0, \ \ \text{for all} \ \ \tau\in \mathbb R.\]
This completes the proof.
\end{proof}

Now we construct a very rich class of FAs of type $a_0$, where $a_0$ is an operation
in $\mathcal O_m$ such that $a_0(j,n)=j$ for any $j,n\in I$. Note that this operation is uniformly distributed.
Then the corresponding Maksimov's multiplication
for arbitrary cubic matrices
\[A=(a_{ijk})_{i,j,k=1}^m, \ \ B=(b_{ijk})_{i,j,k=1}^m, \ \
C=(c_{ijk})_{i,j,k=1}^m,\] is $C=A*_{a_0}B$ where
  \begin{equation}\label{D}
  c_{ijr}=\sum_{k,n=1}^ma_{ijk}b_{knr}.
  \end{equation}

 \begin{thm} Let $g_i(t)$ and $\gamma_{ij}(t)$, $i,j\in I$, be arbitrary functions of
 $t\in \mathbb R$ such that $g_i(t)\ne 0$ and
\begin{equation}\label{gam}
m\sum_{j=1}^m\gamma_{ij}(s)=g_i(s), \ \ \text{for all} \ \, i\in I,  \ s\in \mathbb R.
\end{equation}
Then the cubic matrix
\[\M^{[s,t]}=\left(\frac{\gamma_{ij}(s)}{g_r(t)}\right)_{i,j,r=1}^m\]
generates an FA of type $a_0$ denoted by $\mathcal A^{[s,t]}_6$.
\end{thm}
\begin{proof}
  For $*_{a_0}$, using \eqref{D}, the equation \eqref{KC} can be written as
   \begin{equation}\label{es1}
M_{ijr}^{[s,t]}=\sum_{k,n=1}^m M^{[s,\tau]}_{ijk}M_{knr}^{[\tau,t]}, \ \ \text{for all} \ i,j,r\in I.
\end{equation}
Denote
\begin{equation}\label{fi}
f_{ir}(s,t)=\sum_{j=1}^mM_{ijr}^{[s,t]}.
\end{equation}
Then from \eqref{es1} we get
\begin{equation}\label{es2}
f_{ir}(s,t)=\sum_{k=1}^mf_{ik}(s,\tau)f_{kr}(\tau,t), \ \ \text{for all} \ i,r\in I.
\end{equation}
Consider arbitrary functions $g_i(t)$, $i=1,\dots,m$, such that $g_i(t)\ne 0$.
Then it is easy to see that the system of functional equations \eqref{es2} has the following solution
\[f_{ij}(s,t)=\frac{g_i(s)}{m \, g_j(t)}.\]
Using this equality, by \eqref{fi} and \eqref{es1}, we get
\[ M_{ijr}^{[s,t]}g_r(t)=\frac{1}{m}\sum_{k=1}^m M^{[s,\tau]}_{ijk}g_{k}(\tau), \ \ \text{for all} \ i,j,r\in I.\]
From this equality it is clear that the quantity $ M_{ijr}^{[s,t]}g_r(t)$ should not depend on $t$ and $r$, i.e.
 $M_{ijr}^{[s,t]}g_r(t)=\gamma_{ij}(s)$, for some function $\gamma_{ij}(s)$. Thus
 \[M_{ijr}^{[s,t]}=\frac{\gamma_{ij}(s)}{g_r(t)},  \ \ \text{for all} \  i,j,r\in I.\]
Now the equality \eqref{fi} gives the condition \eqref{gam} on $\gamma_{ij}$.
\end{proof}

\textsl{Remarks on time dynamics}. In general the behavior of $\mathcal A^{[s,t]}_5$ depends on given functions $f_i(t)$ and $g_i(t)$. One can choose these function suitable to an expected property of the dynamics. Here we consider one example to make  $\mathcal A^{[s,t]}_5$ a periodic FA. Consider
\[f_k(t)=2+\sin(kt),\ \ \ g_k(t)=\frac{1}{m^2(2+\sin(kt))}, \  k=1,\dots,m.\] Then the condition \eqref{abm} is satisfied.
Consequently, the matrix
\[\M^{[s,t]}=\left(\frac{2+\sin(is)}{m^2\big(2+\sin(kt)\big)}\right)_{i,j,k=1}^m\]
generates a periodic FA $\mathcal A^{[s,t]}_5$.

The time behavior of $\mathcal A_6^{[s,t]}$ (is similar to the behavior of $\mathcal A^{[s,t]}_5$)
depends on its parameter-functions $g_i(t)$ and $\gamma_{ij}(t)$, $i,j\in I$. \\

Let $S_m$ be the group of permutations on $I$.

Take $a\in \mathcal O_m$ and define an action of $\pi\in S_m$ on $a$ (denoted by $\pi a$) as
\[\pi a(i,j)=\pi a(\pi^{-1}(i),\pi^{-1}(j)), \qquad \text{for all} \ \ i,j\in I.\]

\begin{thm} Let $a\in \mathcal O_m$ and $*_a$ be a Maksimov's multiplication (see \eqref{AB}).
If there exists a permutation $\pi \in S_m$ such that $\pi b=a$,  that is
\begin{equation}\label{abp}
a(j,n)=\pi^{-1} (b(\pi(j),\pi(n))), \qquad \text{for all} \ \ j,n\in I,
\end{equation}
then there is an FA of type $a$ iff there is an FA of type $b$.
\end{thm}
\begin{proof} Assume there is an FA of type $a$, i.e. the equation \eqref{KC} has a solution, say $\M^{[s,t]}=\left(M^{[s,t]}_{ijr}\right)$,
with respect to the multiplication $*_a$. That is
\begin{equation}\label{MM}
 M^{[s,t]}_{ijr}=\sum_{l,n: \, a(l,n)=j}\sum_k  M^{[s,\tau]}_{ilk} M^{[\tau,t]}_{knr}.
\end{equation}
Denote $\pi(i)=i'$, then using \eqref{abp} from \eqref{MM} we get
\[
 M^{[s,t]}_{\pi^{-1}(i')\pi^{-1}(j')\pi^{-1}(r')}=\sum_{l,n: \, b(l',n')=j'}\sum_k  M^{[s,\tau]}_{\pi^{-1}(i')\pi^{-1}(l')\pi^{-1}(k')} M^{[\tau,t]}_{\pi^{-1}(k')\pi^{-1}(n')\pi^{-1}(r')}.
\]
Thus $\tilde\M^{[s,t]}=\left(\tilde M^{[s,t]}_{ijr}\right)$ with
$\tilde M^{[s,t]}_{ijr}= M^{[s,t]}_{\pi^{-1}(i)\pi^{-1}(j)\pi^{-1}(r)}$ is a solution of the equation \eqref{KC}
with respect to the multiplication $*_b$. This completes the proof.
\end{proof}

\subsection{Multiplication of solutions of the equation (\ref{KC})}

In this subsection we assume the following

\textsl{Condition 5}. Assume $\mathfrak C_\mu$ is a commutative algebra.

\begin{thm}\label{tk} Assume $*_\mu$ such that Condition 4 and 5 are satisfied.
Let $\{\M^{[s,t]}\}, \{\mathcal N^{[s,t]}\}\subset \mathfrak C_\mu$ be two families of
$m^3$-matrices which generate two FAs of type $\mu$. Then the matrix
\[\mathcal B^{[s,t]}=\M^{[s,t]}*_\mu \mathcal N^{[s,t]},\]
generates an FA of type $\mu$.
\end{thm}
\begin{proof}
By associativity and commutativity of the multiplication $*_\mu$ of matrices we get
\begin{align*}
\mathcal B^{[s,\tau]}*_\mu \mathcal B^{[\tau,t]}&=\left(\M^{[s,\tau]}*_\mu \mathcal N^{[s,\tau]}\right)*_\mu \left(\M^{[\tau,t]}*_\mu \mathcal N^{[\tau,t]}\right)\\
{}&=\left(\M^{[s,\tau]}*_\mu \M^{[\tau,t]} \right)*_\mu \left(\mathcal N^{[s,\tau]}*_\mu \mathcal N^{[\tau,t]}\right)
=\M^{[s,t]}*_\mu \mathcal N^{[s,t]}=\mathcal B^{[s,t]}.
\end{align*}
\end{proof}

\begin{rk} We note that Theorem~\ref{tk} can be generalized, i.e. under conditions 4 and 5,
let $\{\M_i^{[s,t]}\}\subset \mathfrak C_\mu$, $i=1,\dots,k$, be $k$ families of
$m^3$-matrices which generate $k$ FAs of type $\mu$. Then the matrix
\[\mathcal B^{[s,t]}=\M_1^{[s,t]}*_\mu \M_2^{[s,t]}*_\mu\dots *_\mu \M_k^{[s,t]},\]
generates an FA of type $\mu$.

This formula is useful to construct new FAs by known ones (using above mentioned examples of FAs and other examples given in \cite{LR}).
\end{rk}
\begin{rk} To check our Condition 1--5 for an algebra $\mathfrak C_\mu$ one has to check known conditions (\cite{Ja}) on $\mu=\left(C_{ijk,lnr}^{uvw}\right)$ of the matrix of structural constants. Indeed, one can numerate the elements of the set $\{ijk: i,j,k=1,\dots, m\}$ to write it in the form $\mathcal J=\{1,2,\dots, m^3\}$ then the matrix $\mu$ can be written as usual form: $\mu=(c_{ij}^k)$ where $i,j,k\in \mathcal J$. Then the following conditions are known for an algebra with matrix $\mu$:
\begin{itemize}
\item{\rm Commutative iff}:
\[c_{ij}^k=c_{ji}^k, \ \ \text{for all} \  i,j,k\in \mathcal J.\]

\item{\rm Associative iff}:
\[\sum_{r=1}^nc_{ij}^rc_{rk}^l=\sum_{r=1}^nc_{ir}^lc_{jk}^r, \ \ \text{for all} \  i,j,k,l\in \mathcal J.
\]

\item{\rm Existence of an idempotent element:} this problem is equivalent to the
existence of a fixed point of the map $V \colon  \mathbb R^{m^3}\longrightarrow \mathbb R^{m^3}$, $x \mapsto V(x)= x'$, given by
\[V: x_k'=\sum_{i,j=1}^{m^3}c_{ij}^kx_ix_j, \quad k=1,\dots,m^3.\]
For example, if $\mu$ is a stochastic cubic matrix in  the sense that
\[c_{ij}^k\geq 0, \ \ \sum_{k=1}^{m^3} c_{ij}^k=1, \ \ \text{for all} \ \ i,j,\]
then from known theorems about fixed points it follows that the corresponding operator $V$
has at least one fixed point (i.e. the algebra $\mathfrak C_\mu$ has at least one idempotent element).
 \end{itemize}
\end{rk}
Let us give an example of ACM for which all above mentioned conditions
can be easily checked.

{\bf Example.} Let $\alpha:\mathcal J\times \mathcal J\to \mathcal J$ be a binary operation
on $\mathcal J=\{1,2,\dots, m^3\}\equiv \{ijk: i,j,k=1,\dots,m\}$
and assume that $(\mathcal J, \alpha)$ is a group with the operation $\alpha$, i.e.,
the operation satisfies axioms: associativity, has identity element (denoted by $i_0j_0k_0$)
and each element has an inverse (for each $ijk$ its inverse is denoted by $\overline{ijk}$).

Define  the following multiplication for basis matrices $E_{ijk}$:
\begin{equation}\label{alma}
 E_{ijk}*_\alpha E_{lnr}=E_{\alpha(ijk, lnr)}.
 \end{equation}
Since $(\mathcal J,\alpha)$ is a group it is easy to see that the ACM $\mathfrak C_\alpha$
is associative, with unit matrix $\mathbb I=E_{i_0j_0k_0}$ and each basis element $E_{ijk}$ has
its inverse denoted by $E_{\overline {ijk}}$. Indeed, since $\alpha(ijk, \overline{ijk})=i_0j_0k_0$
 we have
$$E_{ijk}*_\alpha E_{\overline {ijk}}=E_{\alpha(ijk, \overline{lnr})}=E_{i_0j_0k_0}.$$

 Moreover, if the group is commutative then
the algebra $\mathfrak C_\alpha$ is also commutative.

\section{Differential equations for flows of algebras}\label{Sd}

For a continuous-time Markov process, Kolmogorov  derived  forward equations and backward equations, which are a pair of systems of differential equations that describe the time-evolution of the probability transition probabilities $P_{ij}^{[s,t]}$ giving the process \cite{Ko}. For quadratic stochastic processes   derived partial differential
equations with delaying argument were derived. These equations then were used to describe some processes (see \cite{MG}).

In this section we shall derive partial differential equations for the matrices $\mathcal M^{[s,t]}$.

Let $\mathcal M^{[s,t]}=\left(M_{ijk}^{[s,t]}\right)_{i,j,k=1}^m$ generate an FA of type $\mu$. Take a small $h>0$ such that
$s+h<\tau<t$ and using the equation \eqref{KC}, we write
\[\mathcal M^{[s+h,t]}-\M^{[s,t]}=\M^{[s+h,\tau]}*_\mu \M^{[\tau,t]}-\M^{[s,\tau]}*_\mu \M^{[\tau,t]}=\left(\M^{[s+h,\tau]}-\M^{[s,\tau]}\right)*_\mu \M^{[\tau,t]}.\]
Dividing this expression by $h$ and assuming the existence of the limits we obtain
\[\lim_{h\to 0}\frac{\mathcal M^{[s+h,t]}-\M^{[s,t]}}{h}=\lim_{h\to 0}\frac{\M^{[s+h,\tau]}-\M^{[s,\tau]}}{h}*_\mu \M^{[\tau,t]},\]
i.e.
\begin{equation}\label{des}
\frac{\partial}{\partial s}\M^{[s,t]}=\left(\frac{\partial}{\partial s}\M^{[s,\tau]}\right)*_\mu \M^{[\tau,t]},
\end{equation}
here $\frac{\partial}{\partial s}\M^{[s,t]}=\left(\frac{\partial}{\partial s}M_{ijk}^{[s,t]}\right)_{i,j,k=1}^m$.

Similarly with respect $t$ we get
\begin{equation}\label{det}
\frac{\partial}{\partial t}\M^{[s,t]}=\M^{[s,\tau]}*_\mu\left(\frac{\partial}{\partial t}\M^{[\tau,t]}\right).
\end{equation}

Summarize this in the following theorem.
\begin{thm}
If $\M^{[s,t]}$ generates an FA of type $\mu$ then it satisfies the partial differential equations \eqref{des} and \eqref{det}.
\end{thm}

\begin{rk} The equation \eqref{des} is called forward equation and \eqref{det} is called backward equation. We note that
the equations \eqref{Y} are particular cases of these equations. As it was shown above the solution of the equation \eqref{Y} gives only time-homogenous FAs.
 Each matrix of the FAs mentioned in the previous section satisfy equations \eqref{des}, \eqref{det}.
\end{rk}

\section*{ Acknowledgements}

 This work was partially supported  was supported by Agencia Estatal de Investigaci\'on (Spain), grant MTM2016-
79661-P (European FEDER support included, UE) and by Kazakhstan Ministry of Education and Science, grant 0828/GF4: ``Algebras, close to Lie: cohomologies, identities and deformations''.

\end{document}